\documentclass[12pt]{amsart}
\numberwithin{equation}{section}
\usepackage{amsmath,amsthm,amsfonts,amscd,eucal}

\usepackage{xypic}
\usepackage{graphicx}

\usepackage{amssymb}
\def\cf{{\mathcal F}}

\def\ch{{\mathcal H}}

\def\ga{{\mathfrak A}} 
\def\gb{{\mathfrak B}}

\def\gh{{\mathfrak H}}

\def\gar{{\mathfrak R}}

\def\bn{{\mathbb N}}

\def\a{\alpha}

\def\d{\delta}  

\def\eps{\varepsilon}

\def\s{\sigma} 

\def\f{\varphi}  
  
\def\om{\omega} \def\Om{\Omega}

\newtheorem{thm}{Theorem}[section]
\newtheorem{lem}[thm]{Lemma}

\newtheorem{cor}[thm]{Corollary}
\newtheorem{prop}[thm]{Proposition}

\theoremstyle{definition}

\def\di{{\rm d}}

\newcommand{\nn}{\nonumber}

\begin{document}

\title[From discrete to continuous monotone $C^*$-algebras]
{From discrete to continuous monotone $C^*$-algebras via quantum central limit theorems}
\author{Vitonofrio Crismale}
\address{Vitonofrio Crismale\\
Dipartimento di Matematica\\
Universit\`{a} degli studi di Bari\\
Via E. Orabona, 4, 70125 Bari, Italy}
\email{\texttt{vitonofrio.crismale@uniba.it}}
\author{Francesco Fidaleo}
\address{Francesco Fidaleo\\
Dipartimento di Matematica\\
Universit\`{a} degli studi di Roma Tor Vergata\\
Via della Ricerca Scientifica 1, Roma 00133, Italy} \email{{\tt
fidaleo@mat.uniroma2.it}}
\author{Yun Gang Lu}
\address{Yun Gang Lu\\
Dipartimento di Matematica\\
Universit\`{a} degli studi di Bari\\
Via E. Orabona, 4, 70125 Bari, Italy}
\email{\texttt{yungang.lu@uniba.it}}
\date{\today}

\begin{abstract}

\vskip0.1cm\noindent We prove that all finite joint distributions of creation and annihilation operators in Monotone and anti-Monotone Fock spaces
can be realised as Quantum Central Limit of certain operators on a $C^*$-algebra, at least when the test functions are Riemann integrable. Namely, the
approximation is given by weighted sequences of creators and annihilators in discrete monotone $C^*$-algebras, the weight being the above cited test functions.
The construction is then generalised to processes by an invariance principle. \\

\bigskip

\end{abstract}

\maketitle

\section{introduction}

Monotone Fock spaces firstly appeared in the last years of the 1990's in the papers of Lu \cite{Lu}, and Muraki \cite{Mur2}. The investigation on the convergence of sums of the so called self-adjoint position operators on such spaces can be traced out in \cite{Mur}. There, it was defined the so-called monotone independence for random variables in $C^*$- algebraic probability spaces and proved a Central Limit Theorem (CLT for short). Namely, normalised sums of monotonically independent and identically distributed self-adjoint operators
in $C^*$-algebraic probability spaces weakly converge to the arcsine distribution. This result has a direct concrete application as one finds that position operators, seen as self-adjoint random variables in the discrete Monotone $C^*$-algebra, realise the monotone independence.

On the other hand, the study of the asymptotic behaviour of sums of operators (random variables) realising the monotone commutation rules \cite{Boz,{CFL}} sometimes needs the two sums of monotone creators and annihilators
$$
S_{N}^{1}:=\frac{1}{\sqrt{N}}\sum_{i=1}^N a^\dag_i,\,\,\,\,\, S_{N}^{-1}:=\frac{1}{\sqrt{N}}\sum_{i=1}^N a_i,
$$
have to be handled separately. This means that one looks for the limit of the so-called mixed moments, i.e.
$$
\lim_{N\rightarrow+\infty} \f\big(S_{N}^{\eps(1)}\cdots S_{N}^{\eps(m)}\big),\,\,\,\, m\in\mathbb{N},\,\,\,\, \eps(1),\ldots, \eps(m)\in\{-1,1\}
$$
in some $C^*$-algebraic probability space $(\ga,\f)$, $\ga$ being a $C^*$-algebra and $\f$ a state on it. The monotone central limit results above recalled do not cover this case, which it has been already performed in many relevant situations in Quantum Probability under the name of Quantum (or noncommutative) CLT. As an example one can look at \cite{Spe} for the case of $q$-deformed random variables, or at \cite{C} for processes coming from the so-called
$1$-mode type interacting Fock spaces (IFS for short). Furthermore, we mention that in \cite{AHO} such a Quantum CLT was carried on
also for the Haagerup states on the group $C^*$-algebra of the Free Group  with countably many generators by using the so-called singleton condition (see e.g. \cite{ABCL} for details).

It is our aim to deal with this problem in the monotone setting. Indeed, we firstly take $\ga$ a $C^*$-algebra generated by objects concretely realised as creators and annihilators
in the Monotone Fock space over $L^2(J)$, $J$ being a Lebesgue measurable part of $\mathbb{R}$. Then, we consider the $C^*$-algebraic probability space $(\ga, \om_\Om)$, $\Om$ being Fock vacuum
and $\om_\Om$ the relative vector state, together with the vacuum moments of creation and annihilation operators in $(\ga, \om_\Om)$.
More precisely, for $m\in\mathbb{N}$, $f_1,\ldots, f_m\in L^2(J)$,
$\eps(1),\ldots, \eps(m)\in \{-1,1\}$, we look at the collection of the moments $\om_\Om\big(a^{\eps(1)}(f_1)\cdots a^{\eps(m)}(f_m)\big)$, where $a^{-1}(f):=a(f)$, $a^1(f):=a^\dag(f)$ and investigate whether they can arise from a Quantum CLT. Namely, we ask if it is possible to find another $C^*$-algebraic
probability space related in some way to the monotone commutation relations, say $(\gb, \f)$, and some $b_i\in\gb$ with $b^{-1}_i:=b_i$, $b^1_i:=b^\dag_i$, s.t., for any $m\in\mathbb{N}$,
$$
\lim_N \f\big(S_N^{\eps(1)}\cdots S_N^{\eps(m)}\big)=\om_\Om\big(a^{\eps(1)}(f_1)\cdots a^{\eps(m)}(f_m)\big),
$$
where $S_N^{\eps(i)}:=\frac{1}{\sqrt{N}}\sum_{k=1}^N b_k^{\eps(i)}$, $\eps(i)\in\{-1,1\}$. The $C^*$-algebra $\gb$ could be actually obtained by factoring out the free product $C^*$-algebra by the ideal generated by a "concrete" commutator, according to the results established in \cite{CrFid,CrFid2}.

We show that this program can be carried out at least when $J:=[0,1]$ and the test functions $f_i$'s are Riemann integrable maps. The approximating sequence involves sums of creation and annihilation operators generating the concrete unital $C^*$-algebra of the discrete Monotone Commutation Relations, which thus gives $\gb$.
As the limit above contains more informations than those included in the convergence of sums of self-adjoint operators in $\gb$, our result is more general than the CLT in \cite{Mur}, and it seems naturally addressed for applications in Quantum Physics and Applied Mathematics, above all in Quantum Information and Computing.

The paper is organised as follows. In Section \ref{sec2}, we review definitions and some features concerning discrete and continuous Monotone Fock spaces, as well
as the corresponding Monotone $C^*$-algebras. Section \ref{sec3a} is devoted to the main result, i.e. the CLT and is further equipped with some technical results necessary to reach its proof. In the final part, we outline how to achieve a functional
counterpart (Donsker's invariance principle, see \cite{Do}) of the above cited theorem. Notice that the invariance principle in particular entails the weak convergence of position and momentum processes on the Monotone $C^*$-algebra, to the so-called Monotone Brownian motion \cite{Mur2}. In Section \ref{sec4}, we finally highlight the anti-Monotone case obtaining the same kind of results as the Monotone one.

\section{preliminary tools}
\label{sec2}

The first part of the section is directed to recall some useful features related to the monotone discrete and continuous Fock spaces, as well as the corresponding annihilator and creator operators. The Fock spaces we are dealing with are based on the one particle spaces $\ell^2(\bn)=L^2(\bn,\di n)$ (where $\di n$ is the counting measure on $\bn$), and $L^2([0,1],\di t)$ (where $\di t$ is the Lebesgue measure) respectively. The reader is referred to \cite{CFL, Lu, Mur2} for further details.

In the last lines we report
some notations and facts about partitions on a set, necessary for the development of Section \ref{sec3a}.

Fix $k\geq 1$ and denote $I_k:=\{(i_1,i_2,\ldots,i_k)\subset\bn \mid i_1< i_2 < \cdots <i_k\}$. When $k=0$, we take
$I_0:=\{\emptyset\}$, $\emptyset$ being the empty sequence. For each $k$, $\ch_k:=\ell^2(I_k)$ is the Hilbert space giving the $k$-particles
space. The $0$-particle space $\ch_0=\ell^2(\emptyset)$ is identified with the complex scalar field $\mathbb{C}$.
The discrete monotone Fock space is then defined as $\cf_m:=\bigoplus_{k=0}^{\infty} \ch_k$.

Given any increasing sequence $\a=(i_1,i_2,\ldots,i_k)$ with arbitrary length of natural numbers, we denote by $e_{\a}$ the generic element of
the canonical basis of $\cf_m$. The monotone creation and annihilation operators are respectively given, for any $i\in \mathbb{N}$, by
\begin{align*}
&a^\dagger_ie_{\a}:=\left\{
\begin{array}{ll}
e_{(i,i_1,i_2,\ldots,i_k)} & \text{if}\, i< i_1\,, \\
0 & \text{otherwise}\,, \\
\end{array}
\right. \\
&a_i e_{\a}:=\left\{
\begin{array}{ll}
e_{(i_2,\ldots,i_k)} & \text{if}\, k\geq 1\,\,\,\,\,\, \text{and}\,\,\,\,\,\, i=i_1,\\
0 & \text{otherwise}\,, \\
\end{array}
\right.
\end{align*}
where, as usual, $a_i:=a(e_i)$, $a^\dag_i:=a^\dag(e_i)$. One can check that $\|a^\dagger_i\|=\|a_i\|=1$ (see e.g. \cite{Boz}, Proposition 8).
Moreover, $a^\dagger_i$ and $a_i$ are mutually adjoint and satisfy the following relations
\begin{equation}
\label{comrul}
\begin{array}{ll}
  a^\dagger_ia^\dagger_j=a_ja_i=0 & \text{if}\,\, i\geq j\,, \\
  a_ia^\dagger_j=0 & \text{if}\,\, i\neq j\,.
\end{array}
\end{equation}
In addition, the following commutation relation
\begin{equation*}
%\label{iden}
a_ia^\dagger_i=I-\sum_{k=0}^ia^\dagger_k a_k,\quad i\in\bn
\end{equation*}
holds true.

By means of the monotone creation and annihilation operators, one constructs the Monotone $C^*$-algebra $\gar_m$ as the unital
$C^*$-algebra generated $\{a_i\mid i\in\mathbb{N}\}$. Its structure was investigated in \cite{CFL} where the reader
is referred for details.

The incoming part is devoted to recall definitions and some features of continuous Monotone Fock Space.
They allow us to prove several technical statements which, together with some results dealing with discrete monotone Fock space,
are crucial to prove the main result of the following section, that is a Quantum CLT.

Let $J\subseteq \mathbb{R}$ be a Lebesgue measurable set. For any $n\in\mathbb{N}$, $J^n_+$ denote the simplex made by the sequences
$(t_1<t_2<\cdots <t_n)$ of length $n$ of elements in $J$. As usual, we denote $J^0_+=\{\emptyset\}$.
If $\mu_n\equiv\di^n{\bf t}$ is the Lebesgue measure on $\mathbb{R}^n$ for any $n\geq 1$, (hereafter we use the bold as the shorthand notation for vectors), let
$\gh_n$ be the complex Hilbert space $L^2(J^n_+,\mu_n)$. After taking $\mu_0$ as the Dirac unit mass on
$\emptyset$ and $\gh_0:=L^2(J^0_+,\mu_0)=\mathbb{C}\Om$, $\Om$ being the vacuum vector, the continuous Monotone Fock space
is then achieved as $\Gamma_M=\bigoplus_{n=0}^\infty \gh_n$.
The inner product, linear in the second variable, is defined as
$$
\langle f, g \rangle_n=\d_{m,n}\int_{J^n_+}\overline{f(t_1,\ldots,t_n)}g(t_1,\ldots,t_n)\di^n{\bf t},\quad f\in\gh_m,\,g\in\gh_n
$$

The creation operator with test function $f\in \gh$ is defined, for $g\in \gh_n$, $(t_1,t_2,\ldots,t_n)\in J^n_+$, $n\in\mathbb{N}$, $t\in J$, as
\begin{equation*}
[a^\dag(f)g](t,t_1,\ldots t_n):=\left\{
\begin{array}{ll}
f(t)g(t_1,\cdots t_n) & \text{if}\,\, t< t_1, \\
0 & \text{otherwise}\,. \\
\end{array}
\right.
\end{equation*}
Since $\|a^\dag(f)\|\leq \|f\|$ on any $n$-th particles space, this operator can be linearly extended on the whole $\Gamma_M$ by density.

The annihilation
operator with test function $f\in\gh$ is the bounded extension on $\Gamma_M$ of
\begin{equation}
\label{annc}
[a(f)g](t_1,\ldots t_n):=\left\{
\begin{array}{ll}
\int_{t<t_1}\overline{f(t)}g(t,t_1,\cdots t_n)\di t & \text{if}\,\,\, n\geq 1\,, \\
0 & \text{if}\,\,\, n=0\,, \\
\end{array}
\right.
\end{equation}
$g$ being a function in $\gh_{n+1}$ and, as above, $(t_1,t_2,\ldots,t_n)\in J^n_+$, $n\in\mathbb{N}$, $t\in J$.
From \eqref{annc}, it kills the vacuum $\Om$ and one can check $a(f)^*=a^\dag(f)$.
W also note that the construction of continuous monotone Fock space can be obtained as an IFS, see e.g. \cite{Lu2}.

In the last lines of the section, we report some features bridging partitions of a set and sequences of creation and annihilation operators in general IFS.

Let $S$ be a non empty linearly ordered finite set, and
$\s$ a partition of $S$. Namely, $\s=\{V_1,\ldots, V_p\}$ with
$$
V_i\cap V_j=\d_{ij} V_j\,,\quad \cup_{i=1}^pV_i=S\,,
$$
where the $V_h$ are called blocks of the partition $\s$. A partition $\s$ is called crossing if it contains at least two distinct blocks $V_i$ and $V_j$, and elements $v_1,v_2\in V_i$, $w_1,w_2\in V_j$ s.t. $v_1<w_1<v_2<w_2$. Otherwise, it is called non crossing. It is called a pair partition if each block $V_h$ contains exactly two elements. In this case, for any $h$ we write $V_h=(l_h,r_h)$, where $l_h<r_h$, $l_1<l_2<\ldots< l_{|S|/2}$ and $|S|$ is the necessarily even number of elements in $S$.

In what it follows, we typically take $S=\{1,\ldots,m\}$ and denote the set of partitions on it as $P(m)$. Once $m$ is even, say $m=2n$, then $PP(2n)$ and $NCPP(2n)$ will denote the sets of pair partitions and non crossing pair partitions, respectively. As from above, each $\s\in PP(2n)$ can be simply denoted by $(l_h,r_h)_{h=1}^n$. One says that $(l_h,r_h)_{h=1}^n\in NCPP(2n)$ is connected if $l_1=1$ and $r_1=2n$. The cardinalities of $PP(2n)$ and $NCPP(2n)$ are $(2n!)/(2^nn!)$ \cite{An}, and the $n$th Catalan number $\frac{1}{n+1}\binom{2n}{n}$ \cite{Kr}, respectively.

Partitions are a powerful tool when one has to compute mixed moments of creation-annihilation operators w.r.t. the vacuum in general Fock spaces, as we are going to see.
To this aim, take a the Hilbert space $\gh$ and consider a IFS on it (see \cite{CrLu}). It is useful to denote creation and annihilation operators as
$a^1(f):=a^\dag(f)$ and $a^{-1}(f):=a(f)$, for $f\in\gh$, respectively. If $m\geq 1$,
monomials like $a^{\varepsilon(1)}(f_1)\cdots a^{\varepsilon(m)}(f_m)$, where $\eps(1),\ldots, \eps(m)\in\{-1,1\}$ will often be considered for computing joint laws or Central Limit results. In particular, one finds that for the vector state $\om_\Om(\cdot):= \langle \Om, \cdot \Om\rangle$ on the Fock space, the vacuum expectation
\begin{equation}
\label{vacex}
\om_\Om\big(a^{\varepsilon(1)}(f_1)\cdots a^{\varepsilon(m)}(f_m)\big):=\big\langle\Om,a^{\varepsilon(1)}(f_1)\cdots a^{\varepsilon(m)}(f_m)\Om\big\rangle
\end{equation}
is null when $m$ is odd. If instead $m=2n$, both the following conditions necessarily hold for the non vanishing of \eqref{vacex}:
\begin{enumerate}
  \item $\sum_{k=1}^{2n} \varepsilon(k)=0$
  \item $\sum_{j=k}^{2n} \varepsilon(j)\geq 0$, for $k=1,\ldots, 2n$.
\end{enumerate}
This means that the sequences of $\eps$ realising (1) and (2) above give rise to the so-called Dyck words of length $2n$ on the binary alphabet $\{-1,1\}$ \cite{Du}.
As the set giving the totality of strings $\eps$ of type $\{\eps(1),\ldots, \eps(2n)\}$ is naturally identified with $\{-1,1\}^{2n}$, by $\{-1,1\}^{2n}_+$ we denote the family of those $\eps$ representing a Dyck word.

Notice that for $\eps\in\{-1,1\}^{2n}$, one finds $p\in \mathbb{N}$ and a strictly increasing sequence $0\leq l_1< \ldots < l_p\leq 2n$ s.t. $\eps(l_j)=-1$ for any $j$. For the special case $\eps\in\{-1,1\}^{2n}_+$, conditions (1) and (2) above immediately entail $p=n$, $l_1=1$ and $l_n<2n$. In addition, to each $\eps\in\{-1,1\}^{2n}_+$ one can uniquely associate a non crossing pair partition of the set consisting of $2n$ elements. Namely, the first block is obtained just pairing the first consecutive $(-1,1)$ appearing on the string starting from the left. The second pairing will arise by cancelling the two indices previously paired and reproducing the previous scheme to the remaining ones, and so on. For the convenience of the reader, in a table we arrange the situation for $2n=6$.

\vskip1cm
\begin{center}
\begin{tabular}{||c|c||}
  \hline
  \,\,\, Dyck word \,\,\, & \,\,\, pair partition \,\,\,\\
  \hline
  \hline
  \,\,\, $- - - + + +$ \,\,\, & \,\,\,  (34) (25) (16)  \,\,\,\\
  \hline
  \,\,\, $- - + - + +$ \,\,\, & \,\,\, (23) (45) (16) \,\,\,\\
  \hline

   \,\,\, $- + - - + +$ \,\,\, & \,\,\, (12) (45) (36) \,\,\,\\
\hline

\,\,\, $- - + + - +$ \,\,\, & \,\,\,  (23) (14) (56) \,\,\,\\
\hline

\,\,\, $- + - + -+$ \,\,\, & \,\,\,  (12) (34) (56) \,\,\,\\
\hline

\end{tabular}
\end{center}

\vskip1cm

\noindent Thus, a one-to-one correspondence between $\{-1,1\}^{2n}_+$ and $NCPP(2n)$ is realised, and one uses the natural identification $\eps\equiv (l_h,r_h)_{h=1}^n$.

The following notations describe some objects we will use in the successive sections.
For $n,N\in \mathbb{N}$ with $1\leq n\leq N$, we take
$$
\mathfrak{M}_p(2n,N):=\{k:\{1,\ldots,2n\}\rightarrow \{1,\ldots N\} \mid |k^{-1}(j)|=2, j\in R(k)\}
$$
as the set of all $2$-$1$ maps with range $R(k)$ included in $\{1,\ldots N\}$. If further $(l_h,r_h)_{h=1}^n$ is a pair partition on $\{1,\ldots, 2n\}$, by $\mathfrak{M}_p((l_h,r_h)_{h=1}^n,N)$ we denote the collection of $k$ in $\mathfrak{M}_p(2n,N)$ s.t. $k(l_h)=k(r_h)$ for any $h$.

Very often in the sequel, the generic $k(l)$ will be simply denoted as $k_l$ without further mention. We also denote
\begin{align*}
 \mathfrak{M}_p(2n):=&\bigcup_{N\in\bn}\mathfrak{M}_p(2n,N)\,,\\
 \mathfrak{M}_p((l_h,r_h)_{h=1}^n):=&\bigcup_{N\in\bn}\mathfrak{M}_p((l_h,r_h)_{h=1}^n,N)\,.
 \end{align*}

\section{a quantum central limit theorem}
\label{sec3a}

Achieving the main theorem of the section needs some technical results dealing with continuous Monotone Fock Spaces and Monotone $C^*$-algebras,
which we are going to present. For such a purpose, from now on we put $J\equiv[0,1]$.

The next results amount to the computation of vacuum mixed moments for creation and annihilation operators
acting on the Monotone Fock space $\Gamma_M$. Moreover, Lemmata \ref{lemm} and \ref{tocon} are borrowed from \cite{Lu}, Lemmata 3.1 and 3.2 and here reported for the reader's convenience.
\begin{lem}
\label{lemm}
For each $n\in\mathbb{N}$, $t\in[0,1]$, $f\in\gh$ and $G\in\gh_{n+1}$, one has
\begin{equation*}
\begin{split}
[a(f)G](t_1,\ldots t_n)&=\int_0^1\overline{f(t)}G(t,t_1,\cdots t_n)\chi_{[0,t_1)}(t)\di t \\
&=\int_0^1\overline{f(t)}G(t,t_1,\cdots t_n)\chi_{(t,1]}(t_1)\di t.
\end{split}
\end{equation*}
If in particular $G=g\prod_{k=1}^ng_k$ and $1\leq m\leq n$,
\begin{align*}
a(f)&a^\dag(g)a^\dag(g_1)\cdots a^\dag(g_n)\Om
=\int_0^1\di t\overline{f(t)}g(t)a^\dag(g_1\chi_{(t,1]})\cdots a^\dag(g_n)\Om \nn\\
=&\int_0^1\di t\overline{f(t)}g(t)a^\dag(g_1\chi_{(t,1]})\cdots a^\dag(g_m\chi_{(t,1]})a^\dag(g_{m+1})\cdots a^\dag(g_n)\Om\\
=&\int_0^1\di t\overline{f(t)}g(t)a^\dag(g_1\chi_{(t,1]})\cdots a^\dag(g_n\chi_{(t,1]})\Om\,.\nn
\end{align*}
\end{lem}
\begin{lem}
\label{tocon}
Let us take $n\in\mathbb{N}$, $\eps\equiv(l_h,r_h)_{h=1}^n$ be a non crossing connected pair partition on $\{1,\ldots,2n\}$, and $f_1,\ldots, f_{2n}\in\gh$. Then
\begin{align*}
\begin{split}
&\om_\Om\big(a^{\varepsilon(1)}(f_1)\cdots a^{\varepsilon(2n)}(f_{2n})\big) \\
=&\int_0^1 \om_\Om\big(a^{\varepsilon(2)}(f_2\chi_{[0,t)})\cdots a^{\varepsilon(2n-1)}(f_{2n-1}\chi_{[0,t)})\big)\overline{f_1(t)}f_{2n}(t)\di t
\end{split}
\end{align*}
\end{lem}
\begin{prop}
\label{contmom}
Let $n\in\mathbb{N}$, $(l_h,r_h)_{h=1}^n\in NCPP(2n)$ and $f_1,\ldots, f_{2n}\in\gh$. Then
\begin{align}
\begin{split}
\label{mixmomme}
&\om_\Om\big(a^{\varepsilon(1)}(f_1)\cdots a^{\varepsilon(2n)}(f_{2n})\big) \\
=&\int_{J^n}\prod_{1\leq h<m\leq n}\Delta_{t_h,t_m}(r_h,r_m)
\prod_{h=1}^n\overline{f_{l_h}(t_h)}f_{r_h}(t_h)\di t_h\,,
\end{split}
\end{align}
where
$$
\Delta_{t_h,t_m}(r_h,r_m):=\d_{r_m)}(r_h)+\d_{r_h)}(r_m)\chi_{[0,t_h)}(t_m)
$$
and
$$
\d_{j)}(h):=\left\{\begin{array}{cc}
                          1 & \text{if}\,\,\,\, j>h\,, \\
                          0 & {\rm otherwise}\,.
                        \end{array}
                        \right.
$$
\end{prop}
\begin{proof}
We firstly notice that \eqref{mixmomme} holds for $n=1$. Suppose further it is true for each $k<n$.
The assumption on $(l_h,r_h)_{h=1}^n$ gives $r_1=2j$ for some $j$, and the factorisation rule for vacuum mixed moments in interacting Fock spaces
(see e.g. pag. 215 in \cite{CrLu}) yields
\begin{align}
\begin{split}
\label{omom}
& \om_\Om\big(a^{\varepsilon(1)}(f_1)\cdots a^{\varepsilon(2n)}(f_{2n})\big) \\
=& \om_\Om\big(a^{\varepsilon(1)}(f_1)\cdots a^{\varepsilon(2j)}
(f_{2j})\big)\om_\Om\big(a^{\varepsilon(2j+1)}(f_{2j+1})\cdots a^{\varepsilon(2n)}(f_{2n})\big)\,.
\end{split}
\end{align}
We distinguish two cases, namely $j<n$ and $j=n$.

If $j<n$, one firstly notices
$$
\prod_{1\leq h<m\leq j}\Delta_{t_h,t_m}(r_h,r_m)\prod_{j+1\leq h<m\leq n}\Delta_{t_h,t_m}(r_h,r_m)=\prod_{1\leq h<m\leq n}\Delta_{t_h,t_m}(r_h,r_m)\,.
$$
Furthermore, as $(l_h,r_h)_{h=j+1}^n$ is a non crossing pair partition of $\{2j+1,\ldots,2n\}$, the induction assumption applied twice to the r.h.s. of \eqref{omom} gives the thesis.

If $j=n$, Lemma \ref{tocon} gives
\begin{align*}
&\om_\Om\big(a^{\varepsilon(1)}(f_1)\cdots a^{\varepsilon(2n)}(f_{2n})\big) \\
=&\int_0^1 \om_\Om\big(a^{\varepsilon(2)}(f_2\chi_{[0,t)})\cdots a^{\varepsilon(2n-1)}(f_{2n-1}\chi_{[0,t)})\big) \overline{f_1(t)}f_{2n}(t)\di t \\
=&\int_{J^{n}}\prod_{2\leq h<m\leq n}\Delta_{t_h,t_m}(r_h,r_m)
\bigg(\prod_{h=2}^n\overline{f_{l_h}(t_h)}(f_{r_h}\chi_{[0,t_1)})(t_h)\di t_h\bigg) \overline{f_1(t)}f_{2n}(t)\di t,
\end{align*}
where the last equality follows from the induction assumption since $\{(l_h,r_h)\}_{h=2}^{n}$ is non crossing.
Moreover, as $\d_{r_1)}(r_h)=1$ for $h=2,\ldots, n$, one has
\begin{equation*}
\prod_{h=2}^n\d_{r_1)}(r_h)\chi_{[0,t_1)}(t_h)\prod_{2\leq h<m\leq n}\Delta_{t_h,t_m}(r_h,r_m)=\prod_{1\leq h<m\leq n}\Delta_{t_h,t_m}(r_h,r_m)\,,
\end{equation*}
The thesis then follows.
\end{proof}
In our CLT, once fixing $\eps\in\{-1,1\}$, we will deal with the asymptotic behaviour of sums of type
$$
S_N^\eps(a,f):=\frac{1}{\sqrt{N}}
\sum_{k=1}^N a_{k}^\eps f\bigg(\frac{k}{N}\bigg),\,\,\, N\in\mathbb{N}
$$
for $f$ a bounded complex-valued Riemann integrable function on $[0,1]$, i.e. $f\in\mathcal{R}^\infty([0,1])$.
More in detail, we will look for the limit of the vacuum mixed moments of sums as above, i.e.
$$
\lim_{N\rightarrow\infty}\om_\Om\big(S_N^{\eps(1)}(a,f_1)\cdots S_N^{\eps(m)}(a,f_m)\big),
$$
for $m\in\mathbb{N}$, $\eps(1),\ldots,\eps(m)\in\{-1,1\}$ and $f_1, \ldots, f_m\in\mathcal{R}^\infty([0,1])$.
A simpler evaluation of the limit can be achieved after simplifying the computation of such vacuum expectations. The following results are aimed to this goal.
\begin{lem}
\label{ppf}
Let $f_1,\ldots, f_m\in\mathcal{R}^\infty([0,1])$, $m\in\mathbb{N}$. Then, for each $\eps(1),\ldots,\eps(m)\in\{-1,1\}$, one has that
$\om_\Om\big(S_N^{\eps(1)}(a,f_1)\cdots S_N^{\eps(m)}(a,f_m)\big)$
could not vanish only when $m=2n$, $\eps=(\eps(1),\ldots, \eps(2n))\in\{-1,1\}^{2n}_+$ and $k\in \mathfrak{M}_p((l_h,r_h)_{h=1}^n,N)$, $(l_h,r_h)_{h=1}^n$ being
the non crossing pair partition induced by $\eps$. In this case it is equal to
$$
\frac{1}{N^n}\sum_{k\in \mathfrak{M}_p((l_h,r_h)_{h=1}^n,N)}%\sum_{k_{r_1},\ldots,k_{r_n}=1}^n
\om_\Om\big(a^{\eps(1)}_{k_1}\ldots a^{\eps(2n)}_{k_{2n}}\big)\prod_{h=1}^n\overline{f_{l_h}\bigg(\frac{k_{r_h}}{N}\bigg)}f_{r_h}\bigg(\frac{k_{r_h}}{N}\bigg)
$$
\end{lem}
\begin{proof}
Take a sequence of $m$ complex-valued Riemann integrable functions on the unit interval $f_1,\ldots, f_m$ and $\eps(1),\ldots,\eps(m)\in\{-1,1\}$. Then
\begin{align*}
&\om_\Om\big(S_N^{\eps(1)}(a,f_1)\ldots S_N^{\eps(m)}(a,f_m)\big) \\
=& \frac{1}{N^{\frac{m}{2}}}\sum_{k_1,\ldots, k_m=1}^N \om_\Om\big(a^{\eps(1)}_{k_1}\cdots a^{\eps(m)}_{k_m}\big)\prod_{h=1}^m f^\#_h\bigg(\frac{k_h}{N}\bigg), \end{align*}
where
$$
f^\#_h(x):=\left\{\begin{array}{ll}
                      \overline{f_h(x)} & \text{if}\,\, \eps(h)=-1\,,\\
                      f_h(x) & \text{if}\,\, \eps(h)=1\,.
                    \end{array}
                    \right.
$$
As usual, the vacuum mixed moments vanish if $m\neq 2n$ and $\eps$ does not belong to $\{-1,1\}^{2n}_+$.
If $(l_h,r_h)_{h=1}^n$ is the non crossing pair partition induced by
$\eps\in \{-1,1\}^{2n}_+$, the r.h.s. above reduces to
$$
\frac{1}{N^n}\sum_{(l_h,r_h)_{h=1}^n \in NCPP(2n)}\,\sum_{k_{r_1},\cdots, k_{r_{2n}}=1}^N
\om_\Om\big(a^{\eps(1)}_{k_1}\cdots a^{\eps(2n)}_{k_{2n}}\big)\prod_{h=1}^n\overline{f_{l_h}\bigg(\frac{k_{r_h}}{N}\bigg)}f_{r_h}\bigg(\frac{k_{r_h}}{N}\bigg)\,.
$$
To complete the proof, we need to check that for each $j=1,\ldots, n$, one must have $k_{l_j}= k_{r_j}$ to avoid the automatic vanishing of $\om_\Om$.
We first notice that, as a consequence of \eqref{comrul} and Lemma 5.4 in \cite{CFL}, for arbitrary $j,h,k,m$,
\begin{equation}
\label{equation1}
a_ja^\dag_h a_k a^\dag_m=\d_{j,h}\d_{k,m}a_la^\dag_l\,,
\end{equation}
where $\d_{j,h}$ is the Kronecker symbol and $l:=\max\{j,k\}$, and moreover,
\begin{equation}
\label{equation2}
a_ja_ha^\dag_ka^\dag_m=\d_{h,k}\d_{j,m}\d_{j)}(h)a_ja^\dag_j\,.
\end{equation}
We list the cases which appear.
\vskip.3cm
\noindent 1) $(l_h,r_h)_{h=1}^n$ is an interval partition, i.e.
$r_h=l_h+1$ for any $h=1,\ldots, n$. If there exists $j=1,\ldots, n$ s.t. $k_{l_j}\neq k_{r_j}$, then $a^{\eps(1)}_{k_1}\cdots a^{\eps(2n)}_{k_{2n}}$
is null, as from \eqref{comrul}.
\vskip.3cm
\noindent 2) $(l_h,r_h)_{h=1}^n$ is not an interval partition. Here, take an arbitrary $j$ s.t. $r_j\neq l_j+1$. If one proves that
for each $l_j\leq l_h< r_h\leq r_j$, it results $k_{l_h}=k_{r_h}$, the thesis follows after combining this achievement with 1). Indeed, consider the non crossing pair partition $\pi_j$ given by all the blocks $V_h:=(l_h,r_h)$ s.t. $l_j<l_h<r_h<r_j$ for each $h$. Then, this case is split into two sub-cases:
\vskip.3cm
\noindent 2a) in $\pi_j$  one finds $r_h= l_{h}+1$ for each $h$. Then $k_{l_h}=k_{r_h}$ since \eqref{equation1} and further $k_{l_j}=k_{r_j}$ from \eqref{equation2}.
\vskip.3cm
\noindent 2b) $\pi_j$ is not an interval partition. Then, fix a block $V_p:=(l_p,r_p)$ in $\pi_j$ s.t. $r_p\neq l_p+1$ and iterate the same arguments as in 2) and 2a).
\end{proof}
\begin{lem}
\label{mixmom1}
For any $n\in\mathbb{N}$, $\eps\equiv (l_h,r_h)_{h=1}^n\in \{-1,1\}^{2n}_+$ and $k\in \mathfrak{M}_p((l_h,r_h))$, one has
\begin{equation}
\label{chichis}
\om_\Om\big(a^{\eps(1)}_{k_1}\cdots a^{\eps(2n)}_{k_{2n}}\big)=\prod_{1\leq h<m\leq n} \Delta_{k_{r_h},k_{r_m}}(r_h,r_m)\,,
\end{equation}
where
$$
\Delta_{k_{r_h},k_{r_m}}(r_h,r_m):=\d_{r_m)}(r_h)+ \d_{r_h)}(r_m)\d_{k_{r_h})}(k_{r_m})\,.
$$
\end{lem}
\begin{proof}
We firstly prove that, for a generic $p$, if $(l_h,r_h)_{h=1}^p$ is a non crossing connected pair partition and $k\in  \mathfrak{M}_p((l_h,r_h)_{h=1}^p)$, one has
\begin{equation}
\label{chi1}
a_{k_{l_1}}\cdots a^\dag_{k_{r_1}}=\prod_{1\leq h<m\leq p} \Delta_{k_{r_h},k_{r_m}}(r_h,r_m) a_{k_{r_1}}a^\dag_{k_{r_1}}\,.
\end{equation}
Indeed, for $p=2$ the l.h.s. above reduces to
$$
a_{k_{l_1}}a_{k_{l_2}}a^\dag_{k_{r_2}}a^\dag_{k_{r_1}}=\d_{r_1)}(r_2)\d_{k_{r_1})}(k_{r_2})a_{k_{r_1}}a^\dag_{k_{r_1}}
$$
as for Lemma 5.4 in \cite{CFL}.

For $p>2$, take $m$ the greatest index in ${1,\ldots, 2p}$ s.t. $\eps(r_m-1)=-1$ and $\eps(r_m)=\eps(r_m+1)=1$.
Moreover, the non crossing condition gives $r_m-1=l_m$, $r_m+1=r_q$ for some $q<m$. This entails that on the r.h.s. of $a^\dag_{k_{r_m}}$ one finds
only creators, since the partition is connected. As a consequence, again from Lemma 5.4 in \cite{CFL},
\begin{align*}
a_{k_{l_1}}&\cdots a^\dag_{k_{r_1}}=a_{k_{l_1}}\cdots a^{\eps(r_m-2)}_{k_{r_m-2}}a_{k_{l_m}}a^\dag_{k_{r_m}}a^\dag_{k_{r_q}}\cdots a^\dag_{k_{r_1}}\\
&=\d_{r_q)}(r_m)\d_{k_{r_q})}(k_{r_m})a_{k_{l_1}}\cdots a^{\eps(r_m-2)}_{k_{r_m-2}}a^\dag_{k_{r_q}}\cdots a^\dag_{k_{r_1}}\,.\\
\end{align*}
By \eqref{comrul} one knows the indices of the creators in the chain $a^\dag_{k_{r_q}}\cdots a^\dag_{k_{r_1}}$ are in a strict increasing order.
The r.h.s. above thus becomes
\begin{equation*}
\d_{r_q)(r_m)}\bigg[\prod_{r_q\leq r_s\leq r_1}\d_{k_{r_s})}(k_{r_m})\bigg]a_{k_{l_1}}\cdots a^{\eps(r_j-2)}_{k_{r_j-2}}a^\dag_{k_{r_q}}\cdots a^\dag_{k_{r_1}}\,,
\end{equation*}
which is easy to see equivalent to
$$
\bigg[\prod_{h\neq m}\Delta_{{k_{r_h}},k_{r_m}}(r_h,r_m)\bigg] a_{k_{l_1}}\cdots a^{\eps(r_m-2)}_{k_{r_m-2}}a^\dag_{k_{r_q}}\cdots a^\dag_{k_{r_1}}\,.
$$
Thus \eqref{chi1} follows by a standard induction procedure.
It is further easy to check that for $p<q$,
\begin{align}
\begin{split}
\label{chi2}
&\prod_{1\leq h<m\leq q}\Delta_{{k_{r_h}},k_{r_m}}(r_h,r_m)\\
=&\prod_{1\leq h<m\leq p}\Delta_{{k_{r_h}},k_{r_m}}(r_h,r_m)
\prod_{p+1\leq i<j\leq q}\Delta_{{k_{r_i}},k_{r_j}}(r_i,r_j)\,.
\end{split}
\end{align}
Moreover, if $(l_h,r_h)_{h=1}^p$ is an interval partition and $k\in  \mathfrak{M}_p((l_h,r_h)_{h=1}^n)$, one has
\begin{equation}
\label{chi3}
a_{k_{l_1}}a^\dag_{k_{r_1}}\cdots a_{k_{l_p}}a^\dag_{k_{r_p}}=a_{k_{r_m}}a^\dag_{k_{r_m}}\,,
\end{equation}
where $k_{r_m}=\max\{k_{r_1}, \ldots, k_{r_p}\}$ as from Lemma 5.4 in \cite{CFL}.

Now, for any $n\in\mathbb{N}$, $\eps\in \{-1,1\}^{2n}_+$ and $k\in \mathfrak{M}_p((l_h,r_h)_{h=1}^n)$, the vacuum state on the l.h.s. of \eqref{chichis}
can be factored out by suitable vacuum expectations, each of them exhibiting non crossing connected pair partitions (see, e.g. \cite{CrLu}), that is
$$
\om_\Om\big(a^{\eps(1)}_{k_1}\cdots a^{\eps(2n)}_{k_{2n}}\big)=\om_\Om\bigg(\prod_{h=1}^{r_1}a^{\eps(h)}_{k_h}\bigg)
\om_\Om\bigg(\prod_{h=l_{d_1}}^{r_{d_1}}a^{\eps(h)}_{k_h}\bigg)\cdots
\om_\Om\bigg(\prod_{h=l_{d_m}}^{r_{d_m}}a^{\eps(h)}_{k_h}\bigg)\,,
$$
where $m\leq n$ and $d_j\leq n$ are uniquely determined by $\eps$, and $l_{d_1}=r_1+1,l_{d_2}=r_{d_1}+1,\ldots, l_{d_m}=r_{d_{m-1}}+1$.
The thesis then follows after applying \eqref{chi1}-\eqref{chi3} to the r.h.s. above.
\end{proof}
The following technical result contains two parts, each of them applied to discuss the monotone and the anti-monotone (see Section \ref{sec4}) CLT.
\begin{lem}
\label{means}
Let $n\geq 1$, $f$ and $F$ functions belonging to $\mathcal{R}^\infty([0,1])$ and $\mathcal{R}^\infty([0,1])^n$, respectively. Then
\begin{align}
\begin{split}
\label{conv1}
&\lim_{N\rightarrow \infty} \frac{1}{N^{n+1}}\sum_{k=1}^N f\bigg(\frac{k}{N}\bigg)\sum_{k_1,\ldots,k_n=k}^N F\bigg(\frac{k_1}{N},\ldots, \frac{k_n}{N}\bigg) \\
=&\int_0^1 \di t f(t) \int_{[t,1]^n} \di^n{\rm {\bf t}}F(t_1,\ldots,t_n)\,,
\end{split}
\end{align}
\begin{align}
\begin{split}
\label{conv2}
&\lim_{N\rightarrow \infty} \frac{1}{N^{n+1}}\sum_{k=1}^N f\bigg(\frac{k}{N}\bigg)\sum_{k_1,\ldots,k_n=1}^k F\bigg(\frac{k_1}{N},\ldots, \frac{k_n}{N}\bigg) \\
=&\int_0^1\di t f(t) \int_{[0,t]^n}\di ^n{\rm {\bf t}}F(t_1,\ldots,t_n)\,.
\end{split}
\end{align}
\end{lem}
\begin{proof}
The assertion follows after noticing that the left hand sides provide the limit of the Riemann sums for the Riemann integrals on the right hand ones.
\end{proof}
Here ,we find the main result of the paper.
\begin{thm}
\label{clt}
Let $N\in\mathbb{N}$ and $\eps(1),\ldots,\eps(m)\in\{-1,1\}$. Then, for each $m\in\mathbb{N}$ and $f_1,\ldots, f_m\in \mathcal{R}^\infty([0,1])$, one has
$$
\lim_{N\rightarrow\infty}\om_\Om\big(S_N^{\eps(1)}(a,f_1)\cdots S_N^{\eps(m)}(a,f_m)\big)
$$
vanishes for $m$ odd, and for $m=2n$ is equal to
$$
\om_\Om\big(a^{\varepsilon(1)}(f_1)\cdots a^{\varepsilon(2n)}(f_{2n})\big)\,.
$$
\end{thm}
\begin{proof}
Thanks to Lemmata \ref{ppf} and \ref{mixmom1}, one has that
$$
\om_\Om\big(S_N^{\eps(1)}(a,f_1)\cdots S_N^{\eps(m)}(a,f_m)\big)
$$
is null when $m$ is odd or even with $\eps=\eps(1),\ldots,\eps(m)\notin \{-1,1\}^{2n}_+$. For $m=2n$ and $\eps\in \{-1,1\}^{2n}_+$, it is equal to
\begin{equation}
\label{star3}
\frac{1}{N^n}\sum_{k\in \mathfrak{M}_p((l_h,r_h)_{h=1}^n,N)}F(k_{r_1},\ldots, k_{r_n})\prod_{h=1}^n \overline{f_{l_h}\bigg(\frac{k_{r_h}}{N}\bigg)}f_{r_h}\bigg(\frac{k_{r_h}}{N}\bigg)\,,
\end{equation}
where $(l_h,r_h)_{h=1}^n$ is the non crossing pair partition induced by $\eps$ and
$$
F(t_1, \ldots, t_n):=\prod_{1\leq h<m\leq n}
\bigg[\d_{r_m)}(r_h)+ \d_{r_h)}(r_m)\d_{t_h)}(t_m)\bigg]\,.
$$
By Proposition \ref{contmom}, we need to show that \eqref{star3} converges, for $N\rightarrow \infty$, to
\begin{equation}
\label{star33}
\int_{[0,1]^n}\prod_{1\leq h<m\leq n}\Delta_{t_h,t_m}(r_h,r_m)
\prod_{h=1}^n\overline{f_{l_h}(t_h)}f_{r_h}(t_h)\di t_h\,.
\end{equation}
Notice that $F\in\mathcal{R}^\infty([0,1]^n)$, and further
\begin{equation*}
%\label{homo}
F(t_1, \ldots, t_n)=F(ct_1, \ldots, ct_n)
\end{equation*}
for each $c>0$.
As a consequence, \eqref{star3} reduces to
\begin{equation}
\label{limit}
\frac{1}{N^n}\sum_{k\in \mathfrak{M}_p((l_h,r_h)_{h=1}^n,N)}G\bigg(\frac{k_{r_1}}{N},\ldots,\frac{k_{r_n}}{N}\bigg)\,,
\end{equation}
where
$$
G(t_1,\ldots, t_n):=F(t_1,\ldots, t_n)\prod_{h=1}^n\overline{f_{l_h}(t_h)}f_{r_h}(t_h)\,.
$$
As $k\in \mathfrak{M}_p((l_h,r_h)_{h=1}^n,N)$, one has $k_{r_h}\neq k_{r_m}$ when $h\neq m$. If for $j=2,\ldots, n$,
$V_j$ denotes the totality of the sequences $(k_{r_1},\ldots, k_{r_n})\in\{1,\ldots, N\}^n$ containing exactly $j$ identical elements, it follows
\begin{align*}
\sum_{k_{r_1},\ldots, k_{r_n}=1}^N G\bigg(\frac{k_{r_1}}{N},\ldots,\frac{k_{r_n}}{N}\bigg)&=
\sum_{k\in \mathfrak{M}_p((l_h,r_h)_{h=1}^n,N)}G\bigg(\frac{k_{r_1}}{N},\ldots,\frac{k_{r_n}}{N}\bigg)\\
&+\sum_{j=2}^n \sum_{(k_{r_1},\ldots, k_{r_n})\in V_j}G\bigg(\frac{k_{r_1}}{N},\ldots,\frac{k_{r_n}}{N}\bigg)\,.
\end{align*}
If $M:=\max\{\|f_1\|_\infty,\ldots, \|f_{2n}\|_\infty\}$,
for any $j$ it results
$$
\bigg|\sum_{(k_{r_1},\ldots, k_{r_n})\in V_j}G\bigg(\frac{k_{r_1}}{N},\ldots,\frac{k_{r_n}}{N}\bigg)\bigg|\leq \binom{n}{j}M^{2n} N(N-1)\cdots (N-n+j)\,,
$$
and consequently one writes down \eqref{limit} as
\begin{equation*}
\frac{1}{N^n}\sum_{k_{r_1},\ldots, k_{r_n}=1}^N G\bigg(\frac{k_{r_1}}{N},\ldots,\frac{k_{r_n}}{N}\bigg)+ o\bigg(\frac{1}{N}\bigg)\,.
\end{equation*}
Since $G$ is Riemann integrable in $[0,1]^n$, to complete the proof we need to prove that the Riemann sum
\begin{equation}
\label{f1clt}
\frac{1}{N^n}\sum_{k_{r_1},\ldots, k_{r_n}=1}^N G\bigg(\frac{k_{r_1}}{N},\ldots,\frac{k_{r_n}}{N}\bigg)
\end{equation}
converges for $N\rightarrow \infty$ to \eqref{star33}.
Indeed, for $n=1$ the thesis easily follows. When $n=2$, two cases appear, according to whether $\eps$ determines the partition $((1,2),(3,4))$ or $((1,4),(2,3))$.
We see \eqref{f1clt}
in the first case gives
$$
\frac{1}{N^2}\sum_{k_{r_1},k_{r_2}=1}^N\d_{r_2)}(r_1)\overline{f_{l_1}\bigg(\frac{k_{r_1}}{N}\bigg)}f_{r_1}\bigg(\frac{k_{r_1}}{N}\bigg)
\overline{f_{l_2}\bigg(\frac{k_{r_2}}{N}\bigg)}f_{r_2}\bigg(\frac{k_{r_2}}{N}\bigg)\,,
$$
whose limit, for $N\rightarrow \infty$ is
$$
\int_{[0,1]^2}\d_{r_2)}(r_1)\overline{f_{l_1}(t_1)}f_{r_1}(t_1)
\overline{f_{l_2}(t_2)}f_{r_2}(t_2)\di t_1 \di t_2\,.
$$
In the latter case, instead, it reduces to
$$
\frac{1}{N^2}\sum_{k_{r_1},k_{r_2}=1}^N\d_{r_1)}(r_2)\overline{f_{l_1}\bigg(\frac{k_{r_1}}{N}\bigg)}f_{r_1}\bigg(\frac{k_{r_1}}{N}\bigg)
\overline{f_{l_2}\bigg(\frac{k_{r_2}}{N}\bigg)}f_{r_2}\bigg(\frac{k_{r_2}}{N}\bigg)\d_{k_{r_1})}(k_{r_2})\,.
$$
Thus, the limit gives
\begin{align*}
&\int_0^1\d_{r_1)}(r_2)\overline{f_{l_1}(t_1)}f_{r_1}(t_1)\bigg(\int_0^{t_1}\overline{f_{l_2}(t_2)}f_{r_2}(t_2)\di t_2\bigg) \di t_1 \\
=&\int_{[0,1]^2}\d_{r_1)}(r_2)\chi_{[0,t_1)}(t_2)\overline{f_{l_1}(t_1)}f_{r_1}(t_1)\overline{f_{l_2}(t_2)}f_{r_2}(t_2)\di t_1 \di t_2\,,
\end{align*}
the first equality following from \eqref{conv2}.

We now suppose that the result holds for any $k<n$ and extend it to $n$ by induction. The non crossing condition yields that there exists $j\leq n$ s.t. $r_1=2j$. If $j<n$,
the induction assumption directly yields the thesis since
\begin{align*}
 &\prod_{1\leq h<m\leq n}\Delta_{k_h,k_m}(r_h,r_m) \\
 =&\bigg[\prod_{1\leq h<m\leq j} \Delta_{k_h,k_m}(r_h,r_m)\bigg]\bigg[\prod_{j+1\leq h<m\leq n}\Delta_{k_h,k_m}(r_h,r_m)\bigg]\,.
\end{align*}
If $j=n$, (i.e. $(l_h,r_h)_{h=1}^n$ is a connected pair partition), one firstly notices
\begin{equation*}
F\bigg(\frac{k_{r_1}}{N},\ldots, \frac{k_{r_n}}{N}\bigg)
=F\bigg(\frac{k_{r_2}}{N},\ldots, \frac{k_{r_n}}{N}\bigg)\prod_{h=2}^n\d_{r_1)}(r_h)\d_{k_{r_1})}(k_{r_h})
\end{equation*}
As a consequence,
\begin{align}
\begin{split}
\label{star333}
&\frac{1}{N^n}\sum_{k_{r_1},\ldots, k_{r_n}=1}^N G\bigg(\frac{k_{r_1}}{N},\ldots,\frac{k_{r_n}}{N}\bigg) \\
=&\frac{1}{N^n}\sum_{k_{r_1}=1}^N \prod_{h=2}^n\d_{r_1)}(r_h)\overline{f_{l_1}\bigg(\frac{k_{r_1}}{N}\bigg)}f_{r_1}\bigg(\frac{k_{r_1}}{N}\bigg)
\sum_{k_{r_2},\ldots, k_{r_n}=1}^{k_{r_1}} G\bigg(\frac{k_{r_2}}{N},\ldots,\frac{k_{r_n}}{N}\bigg)\,.
\end{split}
\end{align}
Now, denote
$$
H(t_1, \ldots, t_n):=\prod_{1\leq h<m\leq n}\Delta_{t_h,t_m}(r_h,r_m)\,.
$$
Since \eqref{conv2}, the r.h.s. in \eqref{star333} converges for $N\rightarrow \infty$, to
$$
\int_0^1(\overline{f_{l_1}(t_1)}f_{r_1})(t_1)\bigg[\int_{[0,t_1)^{n-1}}H(t_2,\ldots,t_n)\prod_{h=2}^n\d_{r_1)}(r_h)\overline{f_{l_h}(t_h)}f_{r_h}(t_h)\di t_h\bigg]\di t_1
$$
or, equivalently, to
$$
\int_{[0,1]^{n}}\overline{f_{l_1}(t_1)}f_{r_1}(t_1)\bigg[H(t_2,\ldots,t_n)\prod_{h=2}^n\d_{r_1)}(r_h)\overline{f_{l_h}(t_h)}(\chi_{[0,t_1)}
f_{r_h})(t_h)\di t_h\bigg]\di t_1\,,
$$
which is nothing else that \eqref{star33}.
\end{proof}
Theorem \ref{clt} entails as a particular case another central limit-type result shown in the following corollary.
Namely, one finds that the moments of sums of  self-adjoint operators $s_i:=a_i+a^\dag_i$ (i.e. the position operators)
converge, up to rescaling, to the moments of a random variable with arcsine law supported in $[-\sqrt{2},\sqrt{2}]$.
Such a result has been already obtained in \cite{Mur} as a consequence of the CLT for monotone independent random variables.
Here, it is gained only by using the algebraic form of the elements of the monotone $C^*$-algebra.
\begin{cor}
\label{cor}
For each $m\geq 0$,
$$
\lim_{N\rightarrow\infty}\om_\Om\bigg(\bigg(\frac{1}{\sqrt{N}}\sum_{i=1}^Ns_i\bigg)^m\bigg)=\int_{-\sqrt{2}}^{\sqrt{2}}x^m\frac{1}{\pi\sqrt{2-x^2}}dx\,.
$$
\end{cor}
\begin{proof}
Denote symbolically by 1 the function constantly equal to $1$ on $[0,1]$. By expanding the power on the l.h.s. we first note that for $m$ odd the vacuum expectation is null. When $m=2n$, arguing as in Lemma \ref{ppf}, we get
\begin{align*}
\om_\Om\bigg(\bigg(\frac{1}{\sqrt{N}}\sum_{i=1}^Ns_i\bigg)^m&\bigg)
=\frac{1}{N^{\frac{m}{2}}}\sum_{\eps\in\{-1,1\}^m} \sum_{k_1,\ldots, k_m=1}^N
\om_\Om\big(a^{\eps(1)}_{k_1}a^{\eps(2)}_{k_2}\cdots  a^{\eps(m)}_{k_m}\big)\\
=&\frac{1}{N^n}\sum_{\eps\in\{-1,1\}^{2n}_+}\sum_{k_1,\ldots, k_{2n}=1}^N
\om_\Om\big(a^{\eps(1)}_{k_1}a^{\eps(2)}_{k_2}\cdots  a^{\eps(2n)}_{k_{2n}}\big)\\
=&\frac{1}{N^n}\sum_{\eps\in\{-1,1\}^{2n}_+}\sum_{k\in \mathfrak{M}_p((l_h,r_h)_{h=1}^n,N)}\om_\Om\big(a^{\eps(1)}_{k_1}a^{\eps(2)}_{k_2}\cdots  a^{\eps(2n)}_{k_{2n}}\big) \\
=&\sum_{\eps\in\{-1,1\}^{2n}_+}\om_\Om\big(S_N^{\eps(1)}(a,1)\cdots S_N^{\eps(2n)}(a,1)\big)\,.
\end{align*}
By Theorem \ref{clt}, the r.h.s. above converges for $N\rightarrow \infty$ to
\begin{equation*}
\sum_{\eps\in\{-1,1\}^{2n}_+}\om_\Om\big(a^{\eps(1)}(1)\cdots a^{\eps(2n)}(1)\big)\,,
\end{equation*}
which is exactly $\om_\Om\big((a(1)+a^\dag(1))^{m}\big)$
in the sense that the moments above vanish for odds $m$, and the equality holds for $m=2n$.
The proof then follows from Theorem 3.6 in \cite{Lu}. %The anti-monotone case runs parallel.
\end{proof}
We end the section presenting an invariance principle (known as Donsker's invariance principle in classical probability, see \cite{Do})
in the monotone setting. This could be seen as as extension of some results previously obtained in \cite{DeLu, Mur2} concerning
the passage from the quantum random walks constructed in discrete
Monotone Fock spaces, to the canonical position process $\big(a(\chi_{[0,t]})+a^\dag(\chi_{[0,t]})\big)_{t\geq 0}$ in the continuous Monotone Fock space. To deal with the invariance principle, for $N\in \mathbb{N}$,
$0\leq s<t\leq 1$, we get the process
$$
S_{N, [s,t]}^\eps(a,f):=\frac{1}{\sqrt{N}}
\sum_{k=[Ns]+1}^{[Nt]} a_{k}^\eps f\bigg(\frac{k}{N}\bigg),
$$
where for $x\in\mathbb{R}$, $[x]$ is the unique integer s.t. $[x]\leq x< [x]+1$ and $f\in \mathcal{R}^\infty ([0,1])$.
\begin{thm}
Let $N\in\mathbb{N}$. Then, for each $m\in\mathbb{N}$, $\eps(1),\ldots,\eps(m)\in\{-1,1\}$, $f_1,\ldots, f_m\in\mathcal{R}^\infty ([0,1])$, $s_i<t_i$, $i=1,\ldots,m$,
\begin{align}
\begin{split}
\label{fclt}
&\lim_{N\rightarrow\infty}\om_{\Om}\big(S_{N, [s_1,t_1]}^{\eps(1)}(a,f_1)\cdots S_{N, [s_m,t_m]}^{\eps(m)}(a,f_m)\big) \\
=& \bigg\{ \begin{array}{ll}
     0 & {\rm if}\,\,m\,\,{\rm is\,\, odd}\,,\\
     \om_{\Om}\big(a^{\eps(1)}(f_1\chi_{[s_1,t_1]})\cdots a^{\eps(2n)}(f_{2n}\chi_{[s_{2n},t_{2n}]}\big) &  {\rm if}\,\, m=2n\,.
   \end{array}
\end{split}
\end{align}
\end{thm}
\begin{proof}
The proof runs along the same arguments shown in Theorem \ref{clt}. Thus, we limit ourselves to highlights the main changes appearing in.
Having achieved the vanishing
of the l.h.s. in \eqref{fclt} for odd $m$, when $m=2n$ one has that
the result holds for any $k=1$. Suppose further that it is true for $k<n$ and extend it to $n$.
The non crossing condition yields that there exists $j\leq n$ s.t. $r_1=2j$. If $j<n$, since
$$
\frac{1}{N}\sum_{k=1}^{[Nt]}a_k=\frac{t}{[Nt]+\d}\sum_{k=1}^{[Nt]}a_k
$$
for some $\d\in[0,1)$, all runs as in the proof of Theorem \ref{clt}. One has just to take care about
the domain of $\{k_1,\ldots, k_{2n}\}$, appearing more complicated right now. A common domain can be achieved
after noticing that, for $\overline{s}<s$ and $t<\overline{t}$,
$$
\sum_{k=[Ns]+1}^{[Nt]}a_k=\sum_{k=[N\overline{s}]+1}^{[N\overline{t}]}a_k\chi_{\big[[Ns]+1, [Nt]\big]}(k)\,.
$$
If $j=n$, one has
\begin{align*}
&\om_{\Om}\big(S_{N, [s_1,t_1]}^{\eps(1)}(a,f_1)\cdots S_{N, [s_{2n},t_{2n}]}^{\eps(2n)}(a,f_{2n})\big) \\
=& \frac{1}{N^n}\sum_{k_{r_1},\ldots, k_{r_n}=[Ns]+1}^{[Nt]}G_\chi\bigg(\frac{k_{r_1}}{N},\ldots,\frac{k_{r_n}}{N}\bigg)+o\bigg(\frac{1}{N}\bigg)\,,
\end{align*}
where
$$
G_\chi(x_1,\ldots,x_n):=F(x_1,\ldots,x_n)\prod_{h=1}^n\bigg(\overline{f}_{l_h}\chi_{\left[\frac{[Ns_{l_h}]+1}{N}, \frac{[Nt_{l_h}]}{N}\right]}
f_{r_h}\chi_{\left[\frac{[Ns_{r_h}]+1}{N}, \frac{[Nt_{r_h}]}{N}\right]}\bigg)(x_h)
$$
and $s\leq\{s_{l_h}, s_{r_h}\}$, $t\geq\{t_{l_h}, t_{r_h}\}$, $h=1,\ldots,n$. Neglecting the last term above, we can split into the product of
$$
\frac{1}{N^n}\prod_{h=2}^n\sum_{k_{r_1}=[Ns]+1}^{[Nt]}\d_{r_1)}(r_h)\bigg(\overline{f}_{l_1}\chi_{\left[\frac{[Ns_{l_1}]+1}{N}, \frac{[Nt_{l_1}]}{N}\right]}
f_{r_1}\chi_{\left[\frac{[Ns_{r_1}]+1}{N}, \frac{[Nt_{r_1}]}{N}\right]}\bigg)\bigg(\frac{k_1}{N}\bigg)
$$
and
$$
\sum_{k_{r_2},\ldots, k_{r_n}=[Ns]+1}^{k_{r_1}%[Nt]
}G_\chi\bigg(\frac{k_{r_2}}{N},\ldots \frac{k_{r_n}}{N}\bigg)\,,
$$
as in \eqref{star333}.
Again, \eqref{conv2} and Proposition \ref{contmom} allow to complete straightforwardly the proof.
\end{proof}

\section{the anti-monotone case}
\label{sec4}

The structure of anti-Monotone Fock Space is obtained from the Monotone ones just reversing the order of the admissible sequences. More in detail, for $k\geq 1$, denote $I_k^-:=\{(i_1,i_2,\ldots,i_k)\subset\bn \mid i_1> i_2 > \cdots >i_k\}$, where for $k=0$ we take
$I^-_0:=\{\emptyset\}$, $\emptyset$ being the empty sequence. For each $k$, $\ch_k^-:=\ell^2(I_k^-)$ is the Hilbert space giving the $k$-particles
space and the $0$-particle space $\ch_0=\ell^2(\emptyset)$ is identified with the complex scalar field $\mathbb{C}$.
The discrete anti-Monotone Fock space is then defined as $\cf_m^-=\bigoplus_{k=0}^{\infty} \ch_k^-$.

Given a decreasing sequence $\a=(i_1,i_2,\ldots,i_k)$ of natural numbers, $e_{\a}$ is as usual the arbitrary element of
the canonical basis of $\cf_m^-$.

For $i\in\mathbb{N}$,
the anti-Monotone creation and annihilation operators are
\begin{align*}
&b^{\dagger}_ie_{\a}:=\left\{
\begin{array}{ll}
e_{(i,i_1,i_2,\ldots,i_k)} & \text{if}\, i> i_1\,, \\
0 & \text{otherwise}\,, \\
\end{array}
\right. \\
&b_i e_{\a}:=\left\{
\begin{array}{ll}
e_{(i_2,\ldots,i_k)} & \text{if}\, k\geq 1\,\,\,\,\,\, \text{and}\,\,\,\,\,\, i=i_1,\\
0 & \text{otherwise}\,,
\end{array}
\right.
\end{align*}
respectively. The are mutually adjoint with unital norm and satisfy
\begin{equation*}
%\label{comrul1}
\begin{array}{lll}
  b^{\dagger}_ib^{\dagger}_j=b_jb_i=0 & \text{if}\,\, i\leq j\,, \\
  b_ib^{\dagger}_j=0 & \text{if}\,\, i\neq j\,,\\
  b_ib^{\dagger}_i=I-\sum^{+\infty}_{k= i}b^{\dagger}_k b_k\,.
\end{array}
\end{equation*}
Similar to the Monotone setting, for $n\in\mathbb{N}$ one takes $J^n_-$ as the simplex given by the set of sequences
$(t_1>t_2>\cdots >t_n)$ of length $n$ from $J\subseteq \mathbb{R}$, with $J^0_-=\{\emptyset\}$.
If $\mu_n$ is the positive Lebesgue measure on $\mathbb{R}^n$ for any $n\geq 1$, we denote
$\gh_n^-$ the complex Hilbert space $L^2(J^n_-,\mu_n)$ . After taking $\mu_0$ as the Dirac unit mass on
$\emptyset$ and $\gh_0:=L^2(J^0_-,\mu_0)\equiv\mathbb{C}\Om$, $\Om$ being the vacuum vector, the continuous Monotone Fock space
is then achieved as $\Gamma^-_m:=\bigoplus_{n=0}^\infty \gh^-_n$. The inner product is
$$
\langle f, g \rangle_n=\d_{n,m}\int_{J^n_-}\overline{f(t_1,\ldots,t_n)}g(t_1,\ldots,t_n)\di t_1\cdots \di t_n,\,\,\,\, f\in\gh_n,g\in\gh_m\,.
$$
Creation and annihilation operators are
\begin{equation*}
[b^\dag(f)g](t,t_1,\ldots t_n):=\left\{
\begin{array}{ll}
f(t)g(t_1,\cdots t_n) & \text{if}\, t> t_1\,, \\
0 & \text{otherwise}\,, \\
\end{array}
\right.
\end{equation*}
\begin{equation*}
[b(f)g](t_1,\ldots t_n):=\left\{
\begin{array}{ll}
\int_{t>t_1}\overline{f(t)}g(t,t_1,\cdots t_n)\di t & \text{if}\,\,\, n\geq 1\,, \\
0 & \text{if}\,\,\, n=0\,.\\
\end{array}
\right.
\end{equation*}
We report the analogous of Lemmata \ref{lemm} and \ref{tocon}. Although their proof can be obtained after suitable modifications from \cite{Lu}, Lemmata 3.1 and 3.2, we put them below for the convenience of the reader.
\begin{lem}
\label{antilemm}
For $n\in\mathbb{N}$, $t\in[0,1]$, $f\in\gh^-$ and $G\in\gh_{n+1}^-$, one has
\begin{align}
\begin{split}
\label{antia}
[b(f)G](t_1,\ldots t_n)&=\int_0^1\overline{f(t)}G(t,t_1,\cdots t_n)\chi_{(t_1,1]}(t)\di t \\
&=\int_0^1\overline{f(t)}G(t,t_1,\cdots t_n)\chi_{[0,t)}(t_1)\di t\,.
\end{split}
\end{align}
If $G=g\prod_{h=1}^n g_k$ and $1\leq m\leq n$, then
\begin{align}
\begin{split}
\label{antib}
&b(f)b^\dag(g)b^\dag(g_1)\cdots b^\dag(g_n)\Om\\
=&\int_0^1\di t\overline{f(t)}g(t)b^\dag(g_1\chi_{[0,t)})\cdots b^\dag(g_m\chi_{[0,t)})b^\dag(g_{m+1})\cdots b^\dag(g_n)\Om \\
=&\int_0^1\di t\overline{f(t)}g(t)b^\dag(g_1\chi_{[0,t)})\cdots b^\dag(g_n\chi_{[0,t)})\Om\,.
\end{split}
\end{align}
\end{lem}
\begin{proof}
Indeed, \eqref{antib} is a particular case of \eqref{antia}. The second equality in \eqref{antia} follows from the previous one, since $\chi_{(s,1]}(t)=\chi_{[0,t)}(s)$, for $s,t\in [0,1]$.
Moreover, as
$$
\chi_{[0,1]^{n+1}_-}(t_0,\ldots, t_n)=\prod_{k=1}^n \chi_{(t_k,1]}(t_{k-1}),\,\,\, (t_0,\ldots, t_n)\in [0,1]^{n+1}\,,
$$
for $F\in \gh_{n}^-$ one has
\begin{align*}
&\langle F, b(f)G\rangle_n=\langle b^\dag(f)F, G\rangle_{n+1}\\
=& \int_{[0,1]^{n+1}_-}\overline{f(t_0)} \overline{F(t_1,\ldots, t_n)}G(t_0,t_1,\ldots, t_n)\di t_0 \di^n\textbf{t} \\
=& \int_{[0,1]^{n+1}}\prod_{k=1}^n \chi_{(t_k,1]}(t_{k-1}) \overline{f(t_0)} \overline{F(t_1,\ldots, t_n)}G(t_0,t_1,\ldots, t_n)\di t_0 \di^n\textbf{t}\,,
\end{align*}
which is nothing else than
$$
\int_{[0,1]^{n}}\di^n\textbf{t}\prod_{k=2}^n \chi_{(t_k,1]}(t_{k-1})  \overline{F(t_1,\ldots, t_n)}\int_0^1\di t_0\overline{f(t_0)} G(t_0,t_1,\ldots, t_n)\chi_{(t_1,1]}(t_{0})\,.
$$
\end{proof}
\begin{lem}
\label{antitocon}
Let $n\in\mathbb{N}$, $\eps\equiv(l_h,r_h)_{h=1}^n$ be a non crossing connected pair partition on $\{1,\ldots,2n\}$ and $f_1,\ldots, f_{2n}\in\gh$. Then
\begin{align}
\begin{split}
\label{conn1a}
&\om_\Om\big(b^{\varepsilon(1)}(f_1)\cdots b^{\varepsilon(2n)}(f_{2n})\big) \\
=&\int_0^1 \om_\Om\big(b^{\varepsilon(2)}(f_2\chi_{(t,1]})\cdots b^{\varepsilon(2n-1)}(f_{2n-1}\chi_{(t,1]})\big)\overline{f_1(t)}f_{2n}(t)\di t\,.
\end{split}
\end{align}
\end{lem}
\begin{proof}
The thesis follows from an induction procedure on $n$. The case $n=1$ is trivial.
For $n=2$, the unique non crossing connected pair partition is $\pi=((1,4), (2,3))$. The definition of the anti-Monotone annihilator gives
\begin{align*}
\om_\Om \big(b(f_1)b(f_2) b^\dag(f_{3}&s)b^\dag(f_{4})\big)
=\int_{s>t}\overline{f_1(t)}f_4(t)\overline{f_2(s)}f_3(s)\di s  \di t \\
=&\int_0^1\di t \overline{f_1(t)}f_4(t)\int_0^1 \di s\overline{f_2(s)}f_3(s) \chi_{(t,1]}(s) \\
=& \int_0^1 \om_\Om\big(a(f_2\chi_{(t,1]})a^\dag(f_{3}\chi_{(t,1]})\big) \overline{f_1(t)}f_{4}(t)\di t\,.
\end{align*}
We turn to general $n\geq 3$. The connected property gives $l_n\neq 2n-1$, and the non crossing condition entails $r_n=l_n+1$. Thus,
\eqref{antib} yields
\begin{align*}
&b(f_{l_n})b^\dag(f_{l_n+1})\cdots b^\dag(f_{2n})\Om \\
=& \int_0^1 \di t \overline{f_{l_n}(t)}f_{l_n+1}(t)b^\dag(f_{l_n+2}\chi_{[0,t)})\cdots b^\dag(f_{2n}\chi_{[0,t)})\Om\,.
\end{align*}
Consequently, the l.h.s. of \eqref{conn1a} becomes
\begin{equation*}
\int_0^1 \overline{f_{l_n}(t)}f_{l_n+1}(t) \om_\Om \bigg(\prod_{h=1}^{l_n-1}b^{\varepsilon(h)}(f_{h}) \prod_{k=l_n+2}^{2n}b^\dag(f_{k}\chi_{[0,t)})\bigg) \di t\,.
\end{equation*}
Notice that $(l_h,r_h)_{h=1}^{n-1}$ is a connected non crossing pair partition on $\{1,\ldots,2n\}\backslash\{l_n\,r_n\}$,
hence the induction assumption gives
\begin{align*}
&\om_\Om \bigg(\prod_{h=1}^{l_n-1}b^{\varepsilon(h)}(f_{h}) \prod_{k=l_n+2}^{2n}b^\dag(f_{k}\chi_{[0,t)})\bigg)\\
=&\int_0^1\om_\Om \bigg(\prod_{h=2}^{l_n-1}b^{\varepsilon(h)}(f_{h}\chi_{(s,1]})) \prod_{k=l_n+2}^{2n-1}b^\dag(f_{k}\chi_{(s,1]}\chi_{(t,1]})\bigg)\overline{f_1(s)}(f_{2n}\chi_{[0,t)})(s) \di s
\end{align*}
For
$$
\mathbf{C}(s,t):=\prod_{h=2}^{l_n-1}b^{\varepsilon(h)}(f_{h}\chi_{(s,1]}) \prod_{k=l_n+2}^{2n-1}b^\dag(f_{k}\chi_{(s,1]}\chi_{[0,t)})\,,
$$
as $\chi_{(s,1]}(t)=\chi_{[0,t)}(s)$, the l.h.s. of \eqref{conn1a} becomes
\begin{align*}
&\int_0^1 \di s \overline{f_1(s)}f_{2n}(s) \int_0^1\di t\om_\Om(\mathbf{C}(s,t))
\overline{f_{l_n}(t)}(f_{l_n+1}\chi_{(s,1]})(t)  \\
=&\int_0^1\overline{f_1(s)}f_{2n}(s)\om_\Om\big(b^{\varepsilon(2)}(f_2\chi_{(s,1]})\cdots b^{\varepsilon(2n-1)}(f_{2n-1}\chi_{(s,1]})\big) \di s\,,
\end{align*}
where the last equality coming from \eqref{antib}.
\end{proof}
Proposition \ref{contmom} here assumes the following form:
\begin{prop}
\label{anticontmom}
Let $n\in\mathbb{N}$, $(l_h,r_h)_{h=1}^n\in NCPP(2n)$ and $f_1,\ldots, f_{2n}\in\gh$. Then
\begin{align*}
&\om_\Om\big(b^{\varepsilon(1)}(f_1)\cdots b^{\varepsilon(2n)}(f_{2n})\big) \\
=&\int_{J^n}\prod_{1\leq h<m\leq n}\nabla_{t_h,t_m}(r_h,r_m)
\prod_{h=1}^n\overline{f_{l_h}(t_h)}f_{r_h}(t_h)\di t_h\,,
\end{align*}
where
$$
\nabla_{t_h,t_m}(r_h,r_m):=\d_{r_m)}(r_h)+\d_{r_h)}(r_m)\chi_{(t_h,1]}(t_m)\,.
$$
\end{prop}
The proof runs as the Monotone counterpart. Here, one has to exploit Lemma \ref{antitocon} and the identity
$$
\prod_{1\leq h<m\leq n}\nabla_{t_h,t_m}(r_h,r_m)=\prod_{h=2}^n\d_{r_1)}(r_h)\chi_{(t_1,1]}(t_h)\prod_{2\leq h<m\leq n}\nabla_{t_h,t_m}(r_h,r_m)\,.
$$
Concerning the anti-Monotone discrete case, Lemma \ref{mixmom1} can be rephrased in the following way:
\begin{lem}
\label{amixmom1}
For any $n\in\mathbb{N}$, $\eps\equiv (l_h,r_h)_{h=1}^n\in \{-1,1\}^{2n}_+$ and $k\in \mathfrak{M}_p((l_h,r_h))$, one has
\begin{equation*}
\om_\Om\big(b^{\eps(1)}_{k_1}\cdots b^{\eps(2n)}_{k_{2n}}\big)=\prod_{1\leq h<m\leq n} \nabla_{k_{r_h},k_{r_m}}(r_h,r_m)\,,
\end{equation*}
where
$$
\nabla_{k_{r_h},k_{r_m}}(r_h,r_m):=\d_{r_m)}(r_h)+ \d_{r_h)}(r_m)\d_{k_{r_m})}(k_{r_h})\,.
$$
\end{lem}
For the proof, very similar to that of the "twin" result in the previous section, one has just to use the analogue of \eqref{chi2}, that is
\begin{align*}
&\prod_{1\leq h<m\leq q}\nabla_{{k_{r_h}},k_{r_m}}(r_h,r_m)\\
=&\prod_{1\leq h<m\leq p}\nabla_{{k_{r_h}},k_{r_m}}(r_h,r_m)
\prod_{p+1\leq i<j\leq q}\nabla_{{k_{r_i}},k_{r_j}}(r_i,r_j)\,,
\end{align*}
when $p<q$, and the anti-Monotone version of \cite{CFL}, Lemma 5.4
$$
\begin{array}{ll}
  b_kb_jb^\dag_j=\d_{j)}(k)b_k, & b_jb^\dag_jb^\dag_k=\d_{j)}(k)b^\dag_k, \\
  b_jb^\dag_jb_k=b_k, & b^\dag_kb_jb^\dag_j=b^\dag_k,
\end{array}
$$
where the last two equalities hold for $j\geq k$.

The CLT deals, as in the previous section, with the convergence of vacuum mixed expectations for sums of type
$$
S_N^\eps(b,f):=\frac{1}{\sqrt{N}}
\sum_{k=1}^N b_{k}^{\eps}f\bigg(\frac{k}{N}\bigg),\,\,\, N=1,2\dots
$$
with $f\in\mathcal{R}^\infty([0,1])$.
The anti-Monotone CLT can be proven with the help of the above results, \eqref{conv1} in Lemma \ref{means} and using the same arguments exposed in the proof of Theorem \ref{clt}. We leave the details of the proof to the reader.
\begin{thm}
\label{cltbis}
Let $N\in\mathbb{N}$ and $\eps(1),\ldots,\eps(m)\in\{-1,1\}$. Then, for each $m\in\mathbb{N}$ and $f_1,\ldots, f_m\in \mathcal{R}^\infty([0,1])$,
$$
\lim_{N\rightarrow\infty}\om_\Om(S_N^{\eps(1)}(b,f_1)\cdots S_N^{\eps(m)}(b,f_m))
$$
vanishes for $m$ odd and, for $m=2n$ is equal to
$$
\om_\Om(b^{\varepsilon(1)}(f_1)\cdots b^{\varepsilon(2n)}(f_{2n}))\,.
$$
\end{thm}
One finally notices that the anti-Monotone form of Corollary \ref{cor}, giving again the weak convergence of normalised sums of position operators to the arcsine law, as well as the invariance principle, can be easily stated as well.

\bigskip

\textbf{Acknowledgements.} The authors kindly acknowledge the support of the italian INDAM-GNAMPA.

\end{document}